\documentclass[a4paper]{paper}
\usepackage[T1]{fontenc}
\usepackage[utf8]{inputenc}
\usepackage[english]{babel}
\usepackage{amsmath,amsfonts,mathtools}
\usepackage{amsbsy}
\usepackage{amsthm}
\usepackage{bbm}
\usepackage{marvosym}
\usepackage{tikz,tikz-3dplot} 
\usetikzlibrary{arrows}
\usepackage{xspace}
\usepackage{comment}
\usepackage{algorithm}
\usepackage{algorithmic}
\usepackage{enumitem}
\usepackage{blkarray}
\usepackage{booktabs}
\usepackage{hyperref}
\newtheorem{thm}{Theorem}
\newtheorem{lemma}[thm]{Lemma}

\theoremstyle{remark}

\newtheorem{cor}{Corollary}
\newtheorem*{rem}{Remark}

\DeclarePairedDelimiter{\norm}{\lVert}{\rVert}

\usepackage{todonotes}

\begin{document}

\title{A continuation method for computing the multilinear Pagerank}
\author{Alberto Bucci\thanks{M.Sc. graduate at the University of Pisa.} and Federico Poloni\thanks{University of Pisa, Department of Computer Science. \texttt{federico.poloni@unipi.it}. FP is partially supported by INDAM/GNCS and by the University of Pisa's projects PRA\_2017\_05 ``Modelli ed algoritmi innovativi per problemi strutturati e sparsi di grandi  dimensioni'' and PRA\_2020\_61 ``Analisi di reti complesse: dalla teoria alle applicazioni''. }}
\date{}

\maketitle

\begin{abstract}
The multilinear Pagerank model [Gleich, Lim and Yu, 2015] is a tensor-based generalization of the Pagerank model. Its computation requires solving a system of polynomial equations that contains a parameter $\alpha \in [0,1)$. For $\alpha \approx 1$, this computation remains a challenging problem, especially since the solution may be non-unique. Extrapolation strategies that start from smaller values of $\alpha$ and `follow' the solution by slowly increasing this parameter have been suggested; however, there are known cases where these strategies fail, because a globally continuous solution curve cannot be defined as a function of $\alpha$. In this paper, we improve on this idea, by employing a predictor-corrector continuation algorithm based on a more general representation of the solutions as a curve in $\mathbb{R}^{n+1}$. We prove several global properties of this curve that ensure the good behavior of the algorithm, and we show in our numerical experiments that this method is significantly more reliable than the existing alternatives.
\end{abstract}

\section{Introduction}

Set
\[
x^{\otimes m} = \underbrace{x \otimes x \otimes \dots \otimes x}_{\text{$m$ times}},
\]
where $\otimes$ denotes the Kronecker product~\cite[Section~11.4]{Handbook}, and $e = [1,\,1\,\dots,\,1]^T \in \mathbb{R}^n$, the vector of all ones. Let $R \in \mathbb{R}^{n\times n^m}$ be a non-negative matrix such that $e^T R = (e^T)^{\otimes m}$, and $v\in\mathbb{R}^n$ be a \emph{stochastic} vector, i.e., a non-negative vector with $e^T v = 1$.
The \emph{Multilinear Pagerank problem} consists in finding a stochastic solution $x\in\mathbb{R}^n$ to the equation
\begin{equation} \label{mlpr}
x = f_\alpha(x), \quad f_\alpha(x) = \alpha R(x^{\otimes m}) + (1-\alpha) v,
\end{equation}
for a certain $m\in \mathbb{N}, m \geq 2$ and a given value $\alpha \in (0,1)$.

Equation~\eqref{mlpr} is a generalization of the equation behind the well-known Pagerank model~\cite{Pageetal98} (to which it reduces for $m=1$), and has been introduced  in~\cite{DBLP:journals/corr/GleichLY14} as a simplified version of a higher-order Markov chain model with memory. The equation itself has a probabilistic interpretation that was suggested in~\cite{BenGL17}.

Problem~\eqref{mlpr} can be reduced to the computation of Z-eigenvalues of tensors~\cite{lim2005singular,NgQZ09,Qi07}, so some theory and algorithms for that problem can also be applied here.

Various algorithms have been suggested to compute solutions of this equation; see e.g.~\cite{CipRT20, DBLP:journals/corr/GleichLY14,MeiP}. Among the simplest choices we have the fixed-point iteration
\[
x_{k+1} = f_\alpha(x_k), \quad k=0,1,2,\dots,
\]
or the Newton--Raphson method on the function $f_\alpha(x)-x$, i.e.,
\begin{equation} \label{newton}
    x_{k+1} = x_k - (\alpha P_x-I)^{-1} (f_{\alpha}(x_k) - x_k), \quad k=0,1,2,\dots.
\end{equation}
with $\alpha P_x$ the Jacobian matrix of $f_\alpha$.

It is generally recognized that the problem is easier for small $\alpha$, especially for $\alpha \leq \frac1m$. For values of $\alpha$ approaching $1$, its numerical solution is more complicated, and the solution may be non-unique. Sufficient conditions for the uniqueness of solutions have been proposed in literature~\cite{Uniqueness,LiLNV17}; the simplest (but weakest) of them is $\alpha \leq \frac1m$.

In view of this property, various of the algorithms proposed fall in the setting of extrapolation methods, where the a sequence of solutions $x^{(1)}, x^{(2)}, x^{(3)}, \dots$ associated to increasing values of $\alpha^{(1)} < \alpha^{(2)} < \alpha^{(3)} < \alpha^{(k)}$ is computed, and used to obtain an initial guess for a further solution with parameter $\alpha^{(k+1)} > \alpha^{(k)}$. In this way, one can start by solving problems in the `easier regime' $\alpha \approx \frac1m$, and at each step use the previously computed values of $x$ to get a sufficiently accurate initial value.

However, as acknowledged in~\cite{MeiP}, this idea can fail spectacularly, because in some cases the solution of the problem is not given by a continuous function $x(\alpha)$. An example where this is well visible is obtained with the $6\times 6^2$ example \texttt{R6\_3} in the dataset of~\cite{DBLP:journals/corr/GleichLY14} for $\alpha = 0.99$: see Figure~\ref{fig:homoplot} for a plot that shows the behaviour of one entry of $x$ with respect to $\alpha$.
%
%old figure (Bertini)
%------------------
\begin{comment}
\begin{figure}
    \centering
    \includegraphics[width=\textwidth]{homoplot}
    \caption{\protect{The first entry $x_1$ of all solution(s) vs. the parameter $\alpha$ for the example \texttt{R6\_3} in~\cite{DBLP:journals/corr/GleichLY14}. There are three distinct solutions in $\mathcal{Z}_1$ for $\alpha \in $ (approximately), and one solution in the rest of $[0,1]$.  This figure originally appeared in~\cite{MeiP}.}}
    \label{fig:homoplot}
\end{figure}
\end{comment}
%-------------------
\begin{figure}
    \centering
    \includegraphics[width=\textwidth]{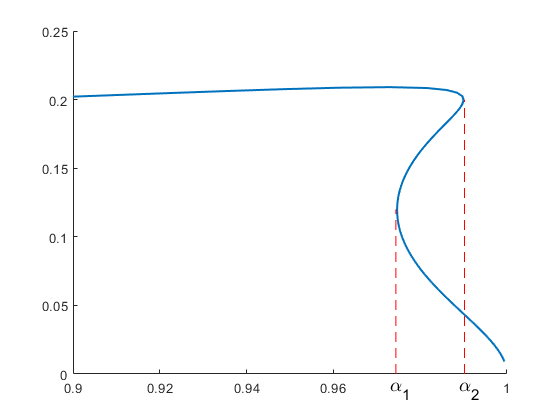}
    \caption{\protect{The first entry $x_1$ of all solution(s) vs. the parameter $\alpha$ for the example \texttt{R6\_3} in~\cite{DBLP:journals/corr/GleichLY14}. There are three distinct stochastic solutions for $\alpha \in [\alpha_1,\alpha_2] \approx [0.9749,\,0.9899]$, and one solution in the rest of $[0,1]$.}}
    \label{fig:homoplot}
\end{figure}
When one attempts to compute solutions for increasing values of $\alpha$, they end up tracking the solution with the largest value of $x_1$ in the figure. When one first surpasses $\alpha = \alpha_2$, this solution is then worthless as an initial value, and the method basically needs to restart without a useful initial guess. In addition, according to the informal description in~\cite[Fig.~9]{DBLP:journals/corr/GleichLY14}, for $\alpha$ slightly larger than $\alpha_2$ the common algorithms are slowed down by the presence of a so-called \emph{pseudo-solution}; they spend many iterations in which the iterates wander about in the set of stochastic vectors before convergence kicks in.

In this paper, we suggest an improved strategy to deal with these problematic examples. We consider the zero set of
\[
H(x,\alpha) := f_\alpha(x) - x =  \alpha R( x^{\otimes m}) + (1-\alpha) v - x, \quad H : \mathbb{R}^n \times \mathbb{R} \to \mathbb{R}^n
\]
as a curve in $\mathbb{R}^{n+1}$, and use extrapolation techniques in the family of \emph{continuation algorithms}~\cite{Allgower:2003:INC:945750} to compute points following this curve. More precisely, the algorithms we use belong to the family of~\emph{predictor-corrector} continuation algorithms.

In cases such as the one in Figure~\ref{fig:homoplot}, the curve makes an \emph{S-bend}, and we cannot consider it anymore (locally) as a parametrized curve with the coordinate $\alpha$ as a parameter: instead, we compute points following the curve, using (implicitly) an arc-length parametrization. Hence, once it reaches $\alpha_2$, our algorithm reduces the value of $\alpha$ tracking the bottom part of the curve in Figure~\ref{fig:homoplot}, reaches $\alpha_1$ again and then increases $\alpha$ a second time.

We give a theoretical contribution, proving that there is indeed a connected solution curve that allows us to track the solution correctly up to $\alpha=1$ in problems without singular points, and then we present a practical algorithm that allows us to solve multilinear Pagerank problem with higher reliability than the existing methods. 

The paper is structured as follows. In Section~\ref{sec:structure}, we recall various results on the properties of solutions, focusing on uniqueness in particular. The results are stated for generic $m$, unlike other references which focus on $m=2$. In Section~\ref{sec:curve}, we focus more closely on the geometrical structure of the solution set as a curve in $\mathbb{R}^{n+1}$. We continue with a general introduction on predictor-corrector methods in Section~\ref{sec:continuation}, and then we present the specific variant of the method that we use for this problem with all its algorithmic details in Section~\ref{sec:thiscontinuation}. In Section~\ref{sec:experiments} we present numerical experiments that prove the effectiveness of this method, and we end with some conclusions in Section~\ref{sec:conclusions}.

\section{Structure of solutions} \label{sec:structure}

We recall various results on the structure of the solution set of~\eqref{mlpr}. Most of these results appear already in previous works (see e.g.~\cite{MeiP}), mostly for $m=2$, but we present them here for completeness in the case of a general $m$.

\begin{thm} \label{theorem 10}
Let $x\in \mathbb{R}^n$ a nonnegative solution of equation~\eqref{mlpr} with $\alpha \in (0;1)$, then $e^T x =1$ or $e^Tx=c_{\alpha}$, where $c_\alpha$ is the other positive solution (besides $z=1$) of the equation $g(z)=0$, with $g(z) := \alpha z^m -z +(1-\alpha)$, and satisfies
\begin{equation}  \label{calpha}
\begin{cases}
c_{\alpha}>1 & \text{if } \alpha<\frac{1}{m},\\
c_{\alpha}=1 & \text{if }\alpha=\frac{1}{m},\\
c_{\alpha}<1 & \text{if } \alpha>\frac{1}{m}.\\
\end{cases}
\end{equation}
\end{thm}

\begin{proof}
Let $x$ be a nonnegative solution of~\eqref{mlpr}.
Multiplying on the left by $e^T$, we get
\[
e^T x = \alpha e^T R(x^{\otimes m}) + (1-\alpha)e^Tv = \alpha (e^T)^{\otimes m} x^{\otimes m} + (1-\alpha) = \alpha (e^T x)^m + (1-\alpha).
\]
Hence $g(e^Tx) = 0$. It remains to prove that the only positive solutions of $g(z)=0$ are $1$ and another one $c_\alpha$ satisfying~\eqref{calpha}. 

By studying the sign of $g'(z)$, one sees that $g(z)$ has a global minimum
\[
z_* = \sqrt[m-1]{\frac{1}{\alpha m}},
\]
for which
\[
\begin{cases}
z_*>1 & \text{if } \alpha<\frac{1}{m},\\
z_*=1 & \text{if }\alpha=\frac{1}{m},\\
z_*<1 & \text{if } \alpha>\frac{1}{m}.\\
\end{cases}
\]
In particular, $g(z_*) \leq g(1) = 0$. The function $g(z)$ is decreasing in $(0,z_*)$ and increasing in $(z_*,\infty)$, hence (unless $z_*=1$) it has two intersections with the $x$-axis, one in $(0,z_*)$ and one in $(z_*,\infty)$.
\end{proof}

One can show that solutions with $e^Tx=1$ and $e^Tx = c_\alpha$ do indeed exist. For any $c \geq 0$, let
\begin{equation} \label{Zc}
    \mathcal{Z}_c := \{x \in \mathbb{R}^n \colon e^T x = c, \, x \geq 0 \},    
\end{equation}
the set of non-negative vectors with fixed entry sum $c$ (in particular, $\mathcal{Z}_1$ is the set of stochastic vectors).

\begin{thm} \label{existence}
Equation~\ref{mlpr} has at least one solution in $\mathcal{Z}_{c_{\alpha}}$, and one in $\mathcal{Z}_1$.
\end{thm}
\begin{proof}

The set $\mathcal{Z}_{c_\alpha}$ is a simplex, and in particular it is compact and convex. Moreover, 
$f_\alpha(\mathcal{Z}_{c_\alpha})\subset \mathcal{Z}_{c_\alpha}$; indeed $\forall y\in\mathcal{Z}_{c_{\alpha}}$
\begin{align*}
e^T f_\alpha(y) & = \alpha e^T R (y^{\otimes m})+(1-\alpha) e^T v\\
         & = \alpha (e^T y)^m  + 1-\alpha\\
         & = \alpha c_\alpha^m + 1-\alpha= c_\alpha, 
\end{align*}
where we have used the fact that $g(c_\alpha)=0$ in the last equality.

Hence by the Brouwer fixed-point theorem~\cite{Brouwer} we can conclude that $f$ has a fixed point in $\mathcal{Z}_{c_\alpha}$, and this fixed point is a solution of~\eqref{mlpr}.

The same proof works after replacing $c_\alpha$ with $1$, since the only property that we have used is that $g(c_\alpha)=0$.
\end{proof}

Moreover, there is a unique solution which achieves the smaller among the possible two values of $e^Tx$. This unique solution is called \emph{minimal solution} in~\cite{MeiP} and other literature on similar matrix equations.
\begin{thm}
For each fixed $\alpha \in [0,1)$, Equation~\eqref{mlpr} has a unique solution in $\mathcal{Z}_{\min(1,c_\alpha)}$.
\end{thm}
\begin{proof}
Consider the iteration
\begin{equation} \label{fixedpoint}
x_0 = 0, \quad  x_{k+1}=f_\alpha(x_k), \quad k=0,1,2,\dots .
\end{equation}
We can show by induction that $x_{k+1}\geq x_k$: the base step $k=0$ is obvious, and
\[
x_{k+1}-x_k= f_\alpha(x_k) - f_\alpha(x_{k-1}) = \alpha R(x_k^{\otimes m} -x_{k-1}^{\otimes m}) \geq 0.
\]
Let $y$ be a solution of~\eqref{mlpr} with $e^Ty = \min(1,c_\alpha)$.
Similarly, we can show by induction that $x_k \leq y$. The base step $k=0$ is again obvious, and
\[
y - x_{k+1} = f_\alpha(y) - f_\alpha(x_k) = \alpha R(y^{\otimes m} - x_k^{\otimes m}) \geq 0.
\]
So the sequence is weakly increasing and bounded, hence it converges to a limit $x_\infty := \lim_{k\to\infty} x_k$.

Passing to the limit, we see that $x_\infty$ is a solution of~\eqref{mlpr}, hence
\[
e^T x_\infty \geq \min(1,c_\alpha) = e^T y.
\]
Again passing to the limit, we see that $x_\infty \leq y$. These two properties together imply that $y = x_\infty$. Thus we have proved that $y=x_\infty$ for each solution $y \in \mathcal{Z}_{\min(1,c_\alpha)}$.
\end{proof}

The following observation seems to be novel. 
\begin{thm}
Let $v$ be a strictly positive vector and let $0 \leq \alpha< 1$.

Any nonnegative solution of~\eqref{mlpr} is strictly positive. \label{Positive}
\end{thm}
\begin{proof}
We have
\[
x_i = e_i^T x = \alpha e_i^T R(x^{\otimes m})+(1-\alpha)e_i^T v =\alpha R_i(x^{\otimes m})+(1-\alpha)v_i > 0,
\]
since $\alpha  R_i(x^{\otimes m})$ and $(1-\alpha)v_i > 0$.
\end{proof}
\begin{cor} \label{cor:solsincube}
In the previous setting, $x\in (0,1)^n$.
\begin{proof}
Since every component of $x$ is strictly positive and the vector is stochastic every component must also be less then 1. 
\end{proof}
\end{cor}
\section{The curve of solutions} \label{sec:curve}

In this section, we argue that the stochastic solutions of~\eqref{mlpr} form a smooth curve (under some regularity assumptions).

We denote with $P_x \in \mathbb{R}^{n\times n}$ the Jacobian matrix of $Rx^{\otimes m}$ with respect to $x$, i.e.,
\begin{align} \label{Px}
P_x = R\frac{\partial x^{\otimes m}}{\partial x} = R(I \otimes x \otimes \dots \otimes x + x \otimes I \otimes x \otimes \dots \otimes x + \dots + x \otimes \dots \otimes x \otimes I).
\end{align}

The Jacobian matrix of $H(x,\alpha)$ is
\begin{equation} \label{Jac}
    J_H[x,\alpha] = \begin{bmatrix}
    \alpha P_x - I & R(x^{\otimes m}) - v
    \end{bmatrix} \in \mathbb{R}^{n\times (n+1)},
\end{equation}
where the first $n\times n$ block column $\frac{\partial H}{\partial x}[x,\alpha] = P_x - I$ contains derivatives with respect to $x$ and the second contains derivatives with respect to $\alpha$.

We let
\[
\Delta_c := \{x \in \mathbb{R}^n : e^Tx = c\}.
\]
Note that this definition differs from that of $\mathcal{Z}_c$ in that we do not require that $x\geq 0$.

We start from a result proving the surjectivity of $\frac{\partial H}{\partial x}[x,\alpha]$ under some assumptions.
\begin{lemma}
Let $(x,\alpha)$ be such that $H(x,\alpha) = 0$, and $x \in \mathcal{Z}_1$.
\begin{enumerate}
    \item For $\alpha < \frac1m$, the matrix $\frac{\partial H}{\partial x}[x,\alpha] = \alpha P_x - I$ has full rank $n$.
    \item For $\alpha = \frac1m$, the matrix $\frac{\partial H}{\partial x}[x,\alpha] = \alpha P_x - I$ has rank $n-1$, and $e^T$ is the unique row vector (up to multiples) such that $e^T \frac{\partial H}{\partial x}[x,\alpha]$.
    \item For $\alpha = \frac1m$, the matrix $\frac{\partial H}{\partial x}[x,\alpha] = \alpha P_x - I$ is surjective as a linear map from $\Delta_0$ to $\Delta_0$.
\end{enumerate}
\end{lemma}
\begin{proof}
We prove the three points.
\begin{enumerate}
    \item Note that the matrix $P_x$ is nonnegative, and that
    \begin{align}
    e^T P_x & = (e^T)^{\otimes m} (I \otimes x \otimes \dots \otimes x + x \otimes I \otimes x \otimes \dots \otimes x + \dots + x \otimes \dots \otimes x \otimes I) \nonumber \\
    &= e^T \otimes e^Tx \otimes \dots \otimes e^Tx + e^Tx \otimes e^T \otimes e^Tx \otimes \dots \otimes e^Tx + \dots + e^Tx \otimes \dots \otimes e^Tx \otimes e^T \nonumber \\
        & = m(e^Tx)^{m-1} e^T. \label{Pxstochasticity}
    \end{align}
    Hence $e^T \alpha P_x < e^T$ when $\alpha < \frac1m$ and $e^Tx=1$. Thus $\alpha P_x$ is a sub-stochastic matrix, and $\alpha P_x - I$ is invertible by the Perron-Frobenius theorem \cite[Section~10.4]{Handbook}.
    \item Again thanks to the relation~\eqref{Pxstochasticity}, the matrix $\alpha P_x$ is (column-)stochastic when $\alpha = \frac1m$ and $e^Tx=1$; hence $e^T (\alpha P_x - I) = 0$. It remains to prove that the rank of $\frac{1}{m} P_x - I$ is exactly $n-1$. Suppose, on the contrary, that it is $n-2$ or lower. Hence the eigenvalue 1 has geometric multiplicity at least $2$ in the column-stochastic matrix $\frac{1}{m} P_x$. It follows from the Perron-Frobenius theory of stochastic matrices (see e.g. \cite[Section~10.4, Fact 1(g)]{Handbook}) that $\frac{1}{m} P_x$ has at least two distinct ergodic classes: that is, there exist two disjoint subsets of indices $E_1, E_2$ such that one can write (reordering the matrix entries and setting $E_3 = (E_1 \cup E_2)^c)$
    \[
    P_x = 
    \begin{blockarray}{cccc}
    & E_1 & E_2 & E_3\\
    \begin{block}{c[ccc]}
    E_1 & P_{E_1E_1} & 0 & P_{E_1E_3}\\
    E_2 & 0 & P_{E_2E_2} & P_{E_2E_3}\\
    E_3 & 0 & 0 & P_{E_3E_3}\\
    \end{block}
    \end{blockarray},
    \]
    with $\frac{1}{m} P_{E_1E_1}$ and $\frac{1}{m} P_{E_2E_2}$ irreducible and stochastic.
    
    Expanding the formula~\eqref{Px}, one sees that the entries of $P_x$ are computed according to
    \begin{multline*}
    P_{i,j} = \sum_{k_1,k_2,k_{m-1}=1}^n  \left(R_{i, (j,k_1,k_2,\dots, k_{m-1})} + R_{i, (k_1,j,k_2,\dots, k_{m-1})} + \dots \right. \\ 
    \left. + R_{i, (k_1,k_2,\dots, k_{m-1},j)} \right) x_{k_1}x_{k_2}\dotsm x_{k_{m-1}},
    \end{multline*}
    where we have used tuples $(j_1,j_2,\dots,j_m) \in \{1,2,\dots,n\}^m$ as indices for the columns of $R$. Let us consider a tuple $(j_1,j_2,\dots,j_m)$ such that $j_1\in E_1$ and $j_2\in E_2$. Since  $x>0$ and $R \geq 0$, we must have
    $R_{i,(j_1,j_2,\dots,j_m)}=0$ for each $i\in E_1$, because otherwise the entry $P_{i,j_2}$ in $P_{E_1E_2}$ would be strictly positive. Analogously, $R_{i,(j_1,j_2,\dots,j_m)}=0$ for each $i \in E_1 \cup E_3$, since otherwise the entry $P_{i,j_1}$ in $P_{E_2E_1}$ or $P_{E_3E_1}$ would be strictly positive. It follows that $R_{i,(j_1,j_2,\dots,j_m)}=0$ for all $i$. However, this implies
    \[
    (e^T R)_{(j_1,j_2,\dots,j_m)} = \sum_{i=1}^n R_{i,(j_1,j_2,\dots,j_m)} = 0,
    \]
    which contradicts $e^T R = (e^T)^{\otimes m}$.
    \item It is sufficient to prove that the map $w \mapsto (\alpha P_x - I)w$ has trivial kernel in $\Delta_0$, i.e., $(\alpha P_x - I)w \neq 0$ for all $w \in \Delta_0 \setminus \{0\}$.
    
    By the Perron--Frobenius theorem, there is a nonnegative vector $z\neq 0$ such that $\alpha P_x z = z$, i.e., $(\alpha P_x - I)z = 0$. Since we know that $\alpha P_x - I$ has rank $n-1$, it follows that $\ker (\alpha P_x - I) = \operatorname{span}(z)$. Since $z$ is nonnegative, $\operatorname{span}(z) \cap \Delta_0 = \{0\}$, hence $\alpha P_x - I$ has trivial kernel in $\Delta_0$. \qedhere
\end{enumerate}
\end{proof}

\begin{thm} \label{thm:curve}
Let $\mathcal{S}=\left\{ (x,\alpha) \mid H(x,\alpha) =0,\, x\in \mathcal{Z}_1\right\}$ and  $\mathcal{S}_{I}=\left\{(x,\alpha)\in \mathcal{S} \mid  \alpha \in I\right\}$.
\begin{enumerate}
    \item $\mathcal{S}_{[0,\frac{1}{m}]}$ is a smooth curve (with boundary).
\end{enumerate}
If moreover $J_H[x,\alpha]$ is surjective as a linear map from $\Delta_0 \times \mathbb{R}$ to $\Delta_0$ for each point $(x,\alpha) \in \mathcal{S}$ with $\alpha \leq 1$, then we have that
\begin{enumerate}[resume]
    \item For some $\varepsilon > 0$,
     $\mathcal{S}_{[-\varepsilon,1+\varepsilon)}$ is union of disjoint smooth curves (with boundary).
    \item The curve $\mathcal{S}^*$ containing $\mathcal{S}_{[0,\frac{1}{m}]}$ reaches the hyperplane $\{\alpha=1\}$.
\end{enumerate}
\end{thm}

\begin{proof}
We prove the three points.
\begin{enumerate}
\item From lemma 1.1 we know that $\frac{\partial H}{\partial x}[x,\alpha]$ is surjective for $0\leq\alpha\leq\frac{1}{m}$; since surjectivity of the Jacobian is an open condition, there exists $\varepsilon > 0$ such that $J_H[x,\alpha]$ is surjective for each point $(x,\alpha)$ with $-\varepsilon <\alpha < 1+\varepsilon$. \\
This allows us to apply the Regular Level Set Theorem~\cite[Section~9.3]{Manifold} to 
$$H_{\big|\Delta_1\times \left(-\varepsilon,\frac1m+\varepsilon\right)}:\Delta_1\times \left(-\varepsilon,\frac1m+\varepsilon\right)\to \Delta_0,$$ obtaining that $H_{\big|\Delta_1\times \left[0,\frac1m\right]}^{-1}(0)$ is union of smooth curves with boundary.

By theorem 3, for each $\alpha\in \left[0,\frac{1}{m}\right]$ there is only one solution, hence there is only one curve in the union and it coincides with $\mathcal{S}_{[0,\frac1m]}$.
\item We argue as in the previous point: since surjectivity of the Jacobian is an open condition, there exists $\varepsilon > 0$ such that $J_H[x,\alpha]$ is surjective for each point $(x,\alpha)$ with $\alpha < 1+\varepsilon$. We can now apply the Regular Level Set Theorem to the whole $\Delta_1\times \left[0,1+\varepsilon\right)$.
\item Let $\mathcal{C}$ denote the hypercube $\left\{(x,\alpha)\in\mathbb{R}^n\times [0,1] \hspace{1mm}\big|\hspace{1mm} 0\leq x_i\leq 1 \right\}$. The intersection $\mathcal{C} \cap \mathcal{S}^* $ is a connected compact 1-manifold with boundary $\partial C \cap \mathcal{S}^*$; such a manifold is diffeomorphic to $\mathbb{S}^1$ or to a closed segment.

It cannot be diffeomorphic to $\mathbb{S}^1$ since otherwise there would be two different branches of $\mathcal{S}^*$ starting from $(v,0)$  (the solution at $\alpha = 0$), contradicting the uniqueness of the solution for $\alpha < \frac1m$.

Being diffeomorphic to a segment, the intersection has two boundary points that must have $\alpha = 0$ or $\alpha = 1$ by Corollary~\ref{cor:solsincube}. Since there is only one solution with $\alpha = 0$, there must be one with $\alpha = 1$. \qedhere

\begin{comment}
Note that its codomain of $H$ is effectively $\Delta_0$, indeed 
\begin{align*}
  e^T H(x,\alpha) &= \alpha e^T R (x^{\otimes m}) + e^T(1-\alpha)v - e^T x \\
  &= \alpha (e^T)^{\otimes m} x^{\otimes m}+ (1-\alpha) -e^T x 
  \\ &= \alpha (e^T x)^m +(1-\alpha)e^T v -(e^T x) = 0.
\end{align*}
% Agreed --- thanks for the fix --federico.
\end{comment}

\end{enumerate}
\end{proof}

\begin{rem}
    We could not find an explicit example in which $J_H[x,\alpha]$ is not surjective as a linear map from $\Delta_0 \times \mathbb{R}$ to $\Delta_0$ for each point $(x,\alpha) \in \mathcal{S}$ with $\alpha \leq 1$. We conjecture that this condition always hold, but we do not have a proof.
\end{rem}

\section{Numerical continuation methods} \label{sec:continuation}

Numerical continuation methods~\cite{Allgower:2003:INC:945750} are a family of numerical methods to trace solution curves defined as the solution set of a system of nonlinear equations of codimension 1, i.e., $G(y)=0$, where $G: \mathbb{R}^{n+1}\to \mathbb{R}^n$.

These methods are often used to compute solutions in a precise region of the curve following the implicit defined curve numerically: one computes an initial solution in a part of the curve where it is easier, and then uses it as an initial step to compute iteratively a sequence of points on the curve, each close to the previous one, with the goal of arriving to a solution in a different part of the curve.

We focus on prediction-correction methods~\cite[Chapter~2]{Allgower:2003:INC:945750}, which are a class of continuation methods structured in two parts:
\begin{enumerate}
    \item (the \emph{predictor step}) Given a starting point $y_k$ on the curve and a step-size $h_k$, we compute a new point $\hat{y}_k$ which is close to the curve (for instance, on its tangent line) and at distance $h_k$ from $y_k$.
    \item (the \emph{corrector step}) We seek a point $y_{k+1}$ on the curve near $\hat{y}_k$ through an iterative method.
        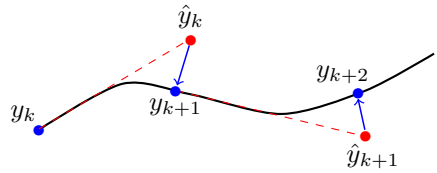
\begin{figure}[h]
    \centering
\begin{tikzpicture}
\draw  [thick]
plot [smooth] coordinates { (-2.2,0.99) (-1.1, 1.6) (-0.2,1.46) (1,1.2) (2.1,1.51) (3,2)}
(-2.2,0.97) node[blue] {\textbullet};
\draw [<-] (-0.38,1.58) -- (-0.22,2.11) [blue, line width=0.2mm];
\draw [<-] (2,1.4) -- (2.09,0.99) [blue, line width=0.2mm];
\draw
(-2.4,1.2) node{$y_k$};
\draw plot (-2.2,1) -- (-0.2,2.2)[dashed,red];
\draw (-0.4,1.51) -- (2.1,0.9)[dashed,red];
\draw (-0.2,2.17) node[red] {\textbullet}
(-0.4,1.49) node[blue] {\textbullet}
(-0.4,1.29) node{$y_{k+1}$}
(-0.2,2.5) node{$\hat{y}_k$};
\draw
(2.1,0.89) node[red] {\textbullet}
(2.0,1.46) node[blue] {\textbullet}
(1.8,1.73) node{$y_{k+2}$}
(2.2,0.65) node{$\hat{y}_{k+1}$};
\end{tikzpicture}
\caption{Illustration of predictor step (red dashed lines) and corrector step (blue lines).}
\end{figure}
    
\end{enumerate}

In order to start the iteration, it is necessary to compute an initial point $y_0$ such that $G(y_0)=0$ and then to alternate the two steps above.

A first example is the strategy suggested in~\cite{MeiP}, which can be interpreted in this framework by setting 
\[
G:\begin{bmatrix}
x \\
\alpha
\end{bmatrix}
\mapsto H(x,\alpha).
\]
As initial step, we choose an initial value $\alpha^{(0)}$ that is close to $\frac1m$, and compute a corresponding $x^{(0)}$ such that $H(x^{(0)},\alpha^{(0)}) = 0$ via either the fixed-point iteration~\eqref{fixedpoint}, or a more efficient method (for instance, the Newton method). Then, at each step $k$,
\begin{enumerate}
    \item (predictor) we choose $\alpha^{(k)} = \alpha^{(k-1)}+h_k$, for a certain step-size $h_k$ (hence $\alpha^{(0)} < \alpha^{(1)} < \dots < \alpha^{(k-1)} < \alpha^{(k)}$ always form an increasing sequence), and use $x^{(0)}, \dots, x^{(k-1)}$ to compute a vector $\hat{x}$ such that $H(\hat{x}, \alpha^{(k)}) \approx 0$. Several variants were tested~\cite{MeiP}, for instance, linear extrapolation, or a Taylor expansion around $x^{(k-1)}$.
    \item (corrector) we use $\hat{x}$ as an initial value for a fixed-point iteration such as~\eqref{fixedpoint} or the Newton method to compute $x^{(k)}$ such that $H(x^{(k)}, \alpha^{(k)})$.
\end{enumerate}

In this setting, at each step we fix a value of $\alpha^{(k)}$ arbitrarily, and then compute the corresponding $x^{(k)}$. Hence essentially we use the variable $\alpha$ to parametrize the curve $\mathcal{S}^*$. This strategy works well in most cases, but runs into trouble when the curve can not be written locally as the graph of a function $x(\alpha)$, i.e., when $\frac{\partial H(x,\alpha)}{\partial x}$ is singular. Figure~\ref{fig:homoplot} is a typical example of this pathological behavior.

Hence, we replace the two steps with more general versions that arise from considering $\mathcal{S}^*$ as a generic curve in $\mathbb{R}^{n+1}$, without singling out $\alpha$ as a preferred parametrization coordinate.

\section{Continuation methods for this problem} \label{sec:thiscontinuation}

We now present the Predictor-Corrector-Newton method, the actual algorithm that we use to solve the Multilinear PageRank problem.

\subsection{Predictor step}

Following~\cite{Allgower:2003:INC:945750}, we use the tangent line to $\mathcal{S}^*$ as a predictor. Given a starting point $(x^{(k)},\alpha^{(k)})$ and a step size $\tau_k$, we look for a point  $(\hat{x}, \hat{\alpha})$ on the tangent line to the solution curve in $(x^{(k)},\alpha^{(k)})$ at distance $\tau_k$ from the starting point.

By the implicit function theorem (see e.g.~\cite[Appendix~B]{Manifold}), the tangent direction is the kernel of the Jacobian matrix $JH[x^{(k)},\alpha^{(k)}]$.

We compute it through a $QR$ decomposition
\[
    JH[x^{(k)},\alpha^{(k)}]^T = [Q\quad q_{n+1}] \begin{bmatrix} R\\ 0 \end{bmatrix}.
\]
Notice that $q_{n+1}\in \mathbb{R}^n$ spans the kernel of $JH[x^{(k)},\alpha^{(k)}]$ (assuming that $JH[x^{(k)},\alpha^{(k)}]$ has full rank).

We also need to make sure that we move on the tangent line in the correct direction to go towards new portions of the curve: to do this, we change the sign of $q_{n+1}$, if needed, so that
    \[
    \left\langle
    \begin{bmatrix}
    x^{(k-1})\\
    \alpha^{(k-1)}
    \end{bmatrix}
    -
    \begin{bmatrix}
    x^{(k)}\\
    \alpha^{(k)}
    \end{bmatrix}
    ,\,
    q_{n+1}
    \right\rangle \geq 0.
    \]
Then we can set
\[
\begin{bmatrix}
\hat{x}\\
\hat{\alpha}
\end{bmatrix}
=
\begin{bmatrix}
x^{(k)}\\
\alpha^{(k)}
\end{bmatrix} + q_{n+1} \tau_k.
\]

We point out that $QR$ factorization can be prohibitively expensive when $n$ is large; in this case, it may be preferable to replace this strategy with linear extrapolation, that is,
\begin{equation}
    (\hat{x}, \hat{\alpha}) = (x^{(k)},\alpha^{(k)}) + \tau_k \frac{(x^{(k)},\alpha^{(k)})-(x^{(k-1)},\alpha^{(k-1)})}{\norm{(x^{(k)},\alpha^{(k)})-(x^{(k-1)},\alpha^{(k-1)})}}.
\end{equation}

% In the experimentation both choices turned out to work fine.
\subsection{Corrector step}

The bulk of our algorithm is the corrector step, where we use Newton's method for underdetermined systems
\begin{equation} \label{pcnewton}
    (x_{j+1},\alpha_{j+1}) = (x_j,\alpha_j)-JH[x_j,\alpha_j]^+H(x_j,\alpha_j), \quad j = 0,1,\dots
\end{equation} to compute a solution of $H(x,\alpha)=0$ close to $(\hat{x},\hat{\alpha})$.

Every iteration requires the computation of $JH[x_j,\alpha_j]^+$, which we compute, again, with QR factorization: if $JH[x_j,\alpha_j]^T = [Q \quad q_{n+1}] \begin{bmatrix}
            R\\
            0
        \end{bmatrix}$, then
\[
   JH[x_j,\alpha_j]^+ = [Q \quad q_{n+1}] \begin{bmatrix} R^{-T} \\ 0 \end{bmatrix}.
\]
% old (wrong) version of this formula:
%\[
%   JH[x_j,\alpha_j]^+ = [R^{-1} \quad 0] \begin{bmatrix} Q^T \\ q_{n+1}^T \end{bmatrix}.
%\]

\subsection{Choice of step size}

We choose an adaptive step size $\tau_k$ with a simplified version of the approach in~\cite[Section~6.1]{Allgower:2003:INC:945750}.

Let $\hat{\delta}$ be the distance between the predicted value $\hat{y}$ and the curve $\mathcal{S}^*$. This distance is approximately equal to the square of the Newton correction at the \emph{first} corrector step $j=1$, up to second-order terms, and it can also be shown to be proportional to the square of the predictor step-length, $\tau_k^2$, i.e.,
\begin{equation} \label{Crelation}
\hat{\delta} \approx \norm{JH[x_1,\alpha_1]^+ H(x_1,\alpha_1)}_1 \approx C \tau_k^2
\end{equation}
for a given constant $C$.
We would like this distance to keep close to a prescribed ``nominal distance'' $\delta$ at each iteration. To this purpose, we compute $C$ from~\eqref{Crelation} (replacing the rightmost $\approx$ with an equal sign), and then choose the next step-size $\tau_{k+1}$ so that $C\tau_{k+1}^2 = \delta$. In addition, some safeguards are taken so that the step-size does not change by more than a factor 2 at each step; see Algorithm~\ref{algo:pc} in the following for details.

\subsection{Final step}

Since the solution curve $\mathcal{S}^*$ reaches $\alpha=1$ (Theorem~\ref{thm:curve}), if the continuation method is successful it will eventually obtain an iteration where $\alpha^{(k)} < \alpha \leq \alpha^{(k+1)}$, where $\alpha > \alpha^{(0)}$ is the desired target value for which we wish to compute a solution $x$ to~\eqref{mlpr}. 

When this happens, we have not reached yet our goal of computing a solution $x$ to~\eqref{mlpr} with the desired target value of the parameter $\alpha$, but we only have two nearby points $(x^{(k)}, \alpha^{(k)})$ and $(x^{(k+1)}, \alpha^{(k+1)})$ on the solution curve $\mathcal{S}^*$. As a final step, we compute $\hat{x}\in\mathbb{R}^n$ such that $(\hat{x}, \alpha)$ is on the segment between these two points (linear interpolation), and use it as the starting point for a final round of the Newton--Raphson method~\eqref{newton} with fixed $\alpha$. We expect this final round to have very fast convergence, since we have identified a suitable starting point near to the solution curve.

Full pseudocode for the algorithm is presented as Algorithm~\ref{algo:pc}.

\begin{algorithm}
\caption{Predictor-Corrector-Newton method} \label{algo:pc}
\vspace{2mm}
\textbf{Parameters} \hspace{2mm} $tol$ \hfill tolerance;\\ 
% \phantom{\textbf{Parameters}} \hspace{2mm} $maxit$ \hfill max. inner iterations;\\ 
\phantom{\textbf{Parameters}} \hspace{2mm} $\alpha_0$ \hfill initial $\alpha$ for curve-following;\\
\phantom{\textbf{Parameters}} \hspace{2mm} $\tau_0$\hfill initial step-size (adaptive);\\
\phantom{\textbf{Parameters}} \hspace{2mm} $\delta $ \hfill nominal distance to curve.\\
\\
\textbf{Subroutines} \hspace{0mm} $\textsf{newton}(x_0,\alpha)$ \hfill Newton's method~\eqref{newton} with fixed $\alpha$. \\
\textbf{Input} \hspace{12mm} $R$, $\alpha$, $v$  \hfill as described above. \\
\textbf{Output} \hspace{9mm} $x$ \hfill a stochastic solution to \eqref{mlpr}.\\
\\
\phantom{\textbf{compute}} \hspace{6mm} $x_0 \leftarrow \textsf{newton}(v, \alpha_0)$\\
\phantom{\textbf{Put}} \hspace{16mm}$y_0\leftarrow(x_0,\alpha_0)$;\\
\phantom{\textbf{Put}} \hspace{16mm}$t_0 \leftarrow e_{n+1}$; \hfill initial direction for curve-following;\\
\phantom{\textbf{Put}} \hspace{16mm}$k \leftarrow 0$; \\%\hfill number of outer iterations;\\
%\phantom{\textbf{Put}} \hspace{16mm}$it = 0$; \hfill number of inner iterations;\\
\phantom{while while wr} \textbf{while} $\alpha_k<\alpha$  \hfill \phantom{a}\\
\phantom{while}\hspace{20mm} $t_{k+1}\leftarrow $ kernel of $JH[y_k]$ (normalized s.t. $\norm{t_{k+1}}=1$);\\
\phantom{while}\hspace{20mm} \textbf{if} $\langle t_k, t_{k+1} \rangle<0$\\
\phantom{while}\hspace{25mm} $t_{k+1}\leftarrow -t_{k+1}$; \hfill ensures $t_{k+1}$ ``points forward'';\\
\phantom{while}\hspace{20mm} \textbf{end}\\
\phantom{while}\hspace{20mm} $\hat{y}_{k+1} \leftarrow y_k+t_{k+1}\tau_k$; \hfill predictor step;\\
\phantom{while}\hspace{20mm} $(\hat{x}_{k+1}, \hat{\alpha}_{k+1}) \leftarrow \hat{y}_{k+1}$;\\
\phantom{while}\hspace{20mm} $j \leftarrow 0$;\\
\phantom{while}\hspace{20mm} \textbf{while} $\norm{H(\hat{y}_{k+1})}_1>tol$  \hfill corrector loop;\\
\phantom{while}\hspace{25mm} $j \leftarrow j+1;$\\
\phantom{while}\hspace{25mm} $d\leftarrow -JH[\hat{y}_{k+1}]^+H(\hat{y}_{k+1})$;\\
\phantom{while}\hspace{25mm} \textbf{if} $j=1$\\
\phantom{while}\hspace{29mm} $f\leftarrow \sqrt{\frac{\norm{d}_1}{\delta}}$; \hfill deceleration factor as in \cite[Sec.~6.1]{Allgower:2003:INC:945750};\\
\phantom{while}\hspace{29mm} \textbf{if} $f>2$ \hfill predictor step too large; \\
\phantom{while}\hspace{33mm} $\tau_k\leftarrow \frac{\tau_k}{2}$;\\
\phantom{while}\hspace{33mm} \textbf{back} to predictor step; \\
\phantom{while}\hspace{29mm} \textbf{end}\\
\phantom{while}\hspace{25mm} \textbf{end} \\ \phantom{while}\hspace{25mm} $\hat{y}_{k+1}\leftarrow \hat{y}_{k+1}+d; $\\
% \phantom{while}\hspace{25mm} $it \leftarrow it+j$;\\
% \phantom{while}\hspace{25mm} \textbf{if} $it > maxit$\\
% \phantom{while}\hspace{29mm} \textbf{failure}; \hfill reached $maxit$ iterations;\\
\phantom{while}\hspace{25mm} \textbf{end}\\
\phantom{while}\hspace{20mm} \textbf{end}\\
\phantom{while}\hspace{20mm} $y_{k+1}\leftarrow \hat{y}_{k+1}$;\\
\phantom{while}\hspace{20mm} $(x_{k+1}, \alpha_{k+1})\leftarrow y_{k+1}$;\\
\phantom{while}\hspace{20mm} $f\leftarrow \max{(f, \frac12)}$;\\
\phantom{while}\hspace{20mm} $\tau_{k+1}\leftarrow \frac{\tau_k}{f}$; \hfill choose next step-size;\\
\phantom{while}\hspace{20mm} $\tau_{k+1}\leftarrow \min{(\tau_{k+1},5 \tau_0)}$;\\
\phantom{while}\hspace{20mm} $k \leftarrow k+1$;\\
\phantom{while while wr} \textbf{end}\\
\phantom{while while wr} $\eta \leftarrow \frac{\alpha-\alpha_{k-1}}{\alpha_k-\alpha_{k-1}}$; \hfill linear interpolation;\\
\phantom{while while wr} $\hat{y}\leftarrow y_{k-1}+\eta(y_k-y_{k-1}); $\\
\phantom{while while wr} $(\hat{x},\hat{\alpha}) \leftarrow \hat{y}$; \hfill it must hold that $\hat{\alpha}=\alpha$;\\
\phantom{while while wr} $x\leftarrow \textsf{newton}(\hat{x}, \alpha)$.
\vspace{2mm}
\end{algorithm}

\newpage

\section{Numerical experiments} \label{sec:experiments}

We compare the Predictor-Corrector-Newton algorithm (PC-N, Algorithm~\ref{algo:pc}) with two different methods: the Newton method (N) with $x_0=(1-\alpha)v$ and the Perron-Newton method (PN-IMP). These algorithms are implemented as described in ~\cite{DBLP:journals/corr/GleichLY14} and ~\cite{MeiP} respectively, and have been chosen because they were the better performers among the algorithms considered in those two papers.

We investigate their performance on two sets of problems: the benchmark set used in ~\cite{DBLP:journals/corr/GleichLY14}, and a set of $100\,000$ random matrices $R\in\mathbb{R}^{5 \times 5^2}$. Each of the latter ones has been created starting from the zero matrix and then, for each $j$, setting $R_{ij}=1$, where $i\in \{1,2,3,4,5\}$ is chosen at random (uniformly and independently).

To better compare iteration counts for algorithms that are based on nested iterations, in the following we define one `iteration' to be an inner iteration, that is, anything that requires a matrix factorization on a dense linear algebra subproblem:
\begin{center}
\begin{tabular}{ll}
\toprule
Algorithm & What is counted as one iteration\\
\midrule
 N & a Newton iteration; \\ 
 PN-IMP  & a Newton or a Perron-Newton iteration;\\
 PC-N & a Newton iteration, or a predictor step, or a corrector step.\\
 \bottomrule
\end{tabular}
\end{center}

In all the experiments we have set the tolerance $tol=\sqrt{\mathbf{u}}$, where $\mathbf{u}$ is the machine precision, the maximum number of iterations $maxit=10\,000$ and $v=\frac{1}{n}e$. In PC-N, we used $\tau_0=0.01$ and $\delta = 0.1$.

The experiments have been run on Matlab R2020b on a computer equipped with a 64-bit Intel core i7-4700MQ processor.

Note that neither the number of iterations nor the CPU time are particularly indicative, alone, of the true performance to expect from the algorithms on large tensors: indeed, the iterations in the different methods amount to different quantities of work, and the CPU time may scale differently for each algorithm when one switches to linear algebra primitives better suited to large-scale problems. Rather, we focus here on the reliability of the algorithms, i.e., the number of problems that they can solve successfully.

In the first experiment, we ran each algorithm on every tensor of the benchmark set, checking the number of iterations and the CPU times in seconds required for the solution of the problem. Performance profiles (see e.g.~\cite[Section~22.4]{matlabguide} for an introduction to this type of plot) are reported in Figures~\ref{fig:perfprof1}--\ref{fig:perfprof3}.

\begin{figure}
\includegraphics[width=12cm]{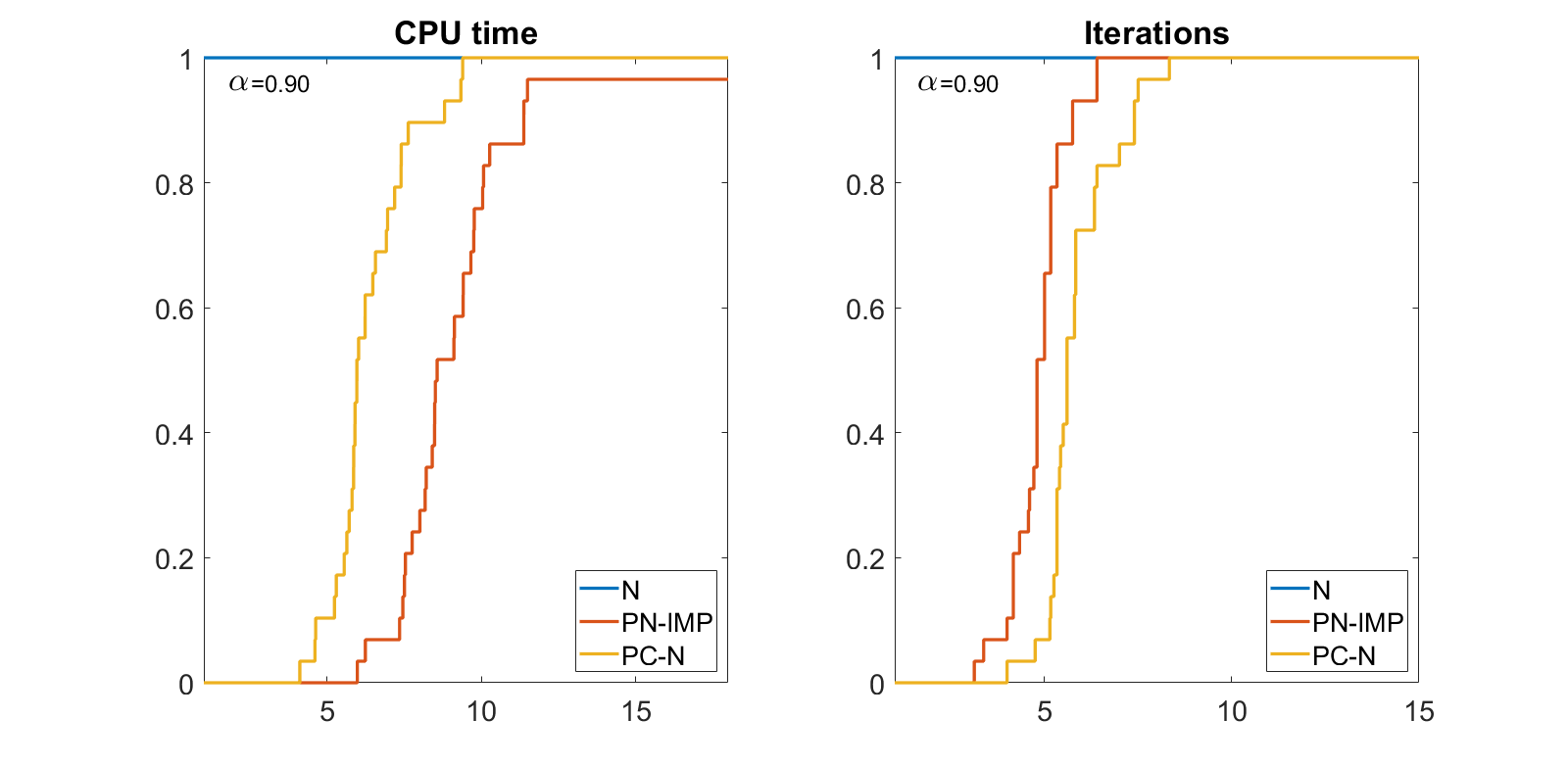}
\caption{Performance profiles for the 29 benchmark tensors and $\alpha= 0.90$.} \label{fig:perfprof1}
\end{figure}
\begin{figure}
\includegraphics[width=12cm]{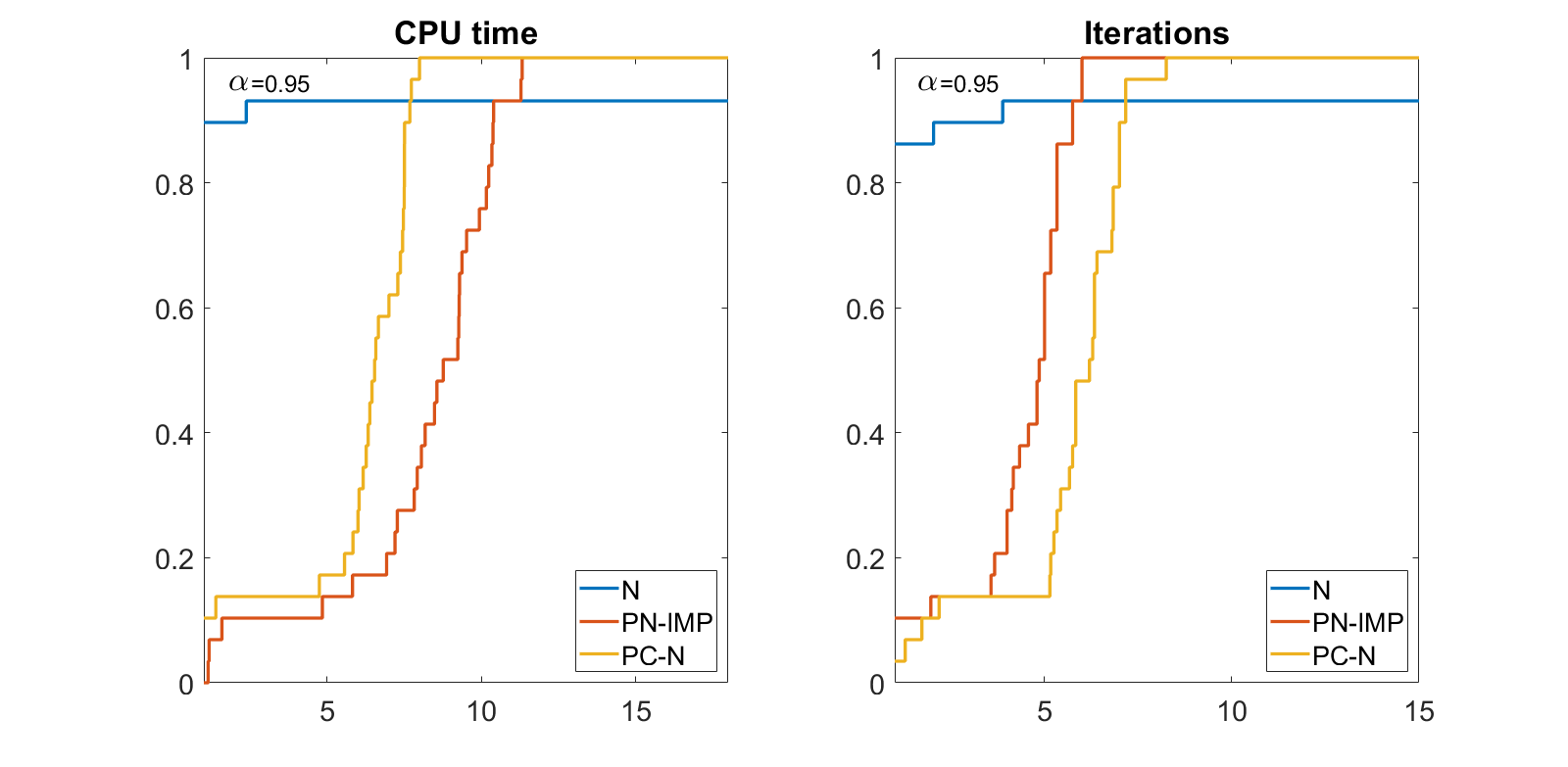}
\caption{Performance profiles for the 29 benchmark tensors and $\alpha= 0.95$.}
\end{figure}
\begin{figure}
\includegraphics[width=12cm]{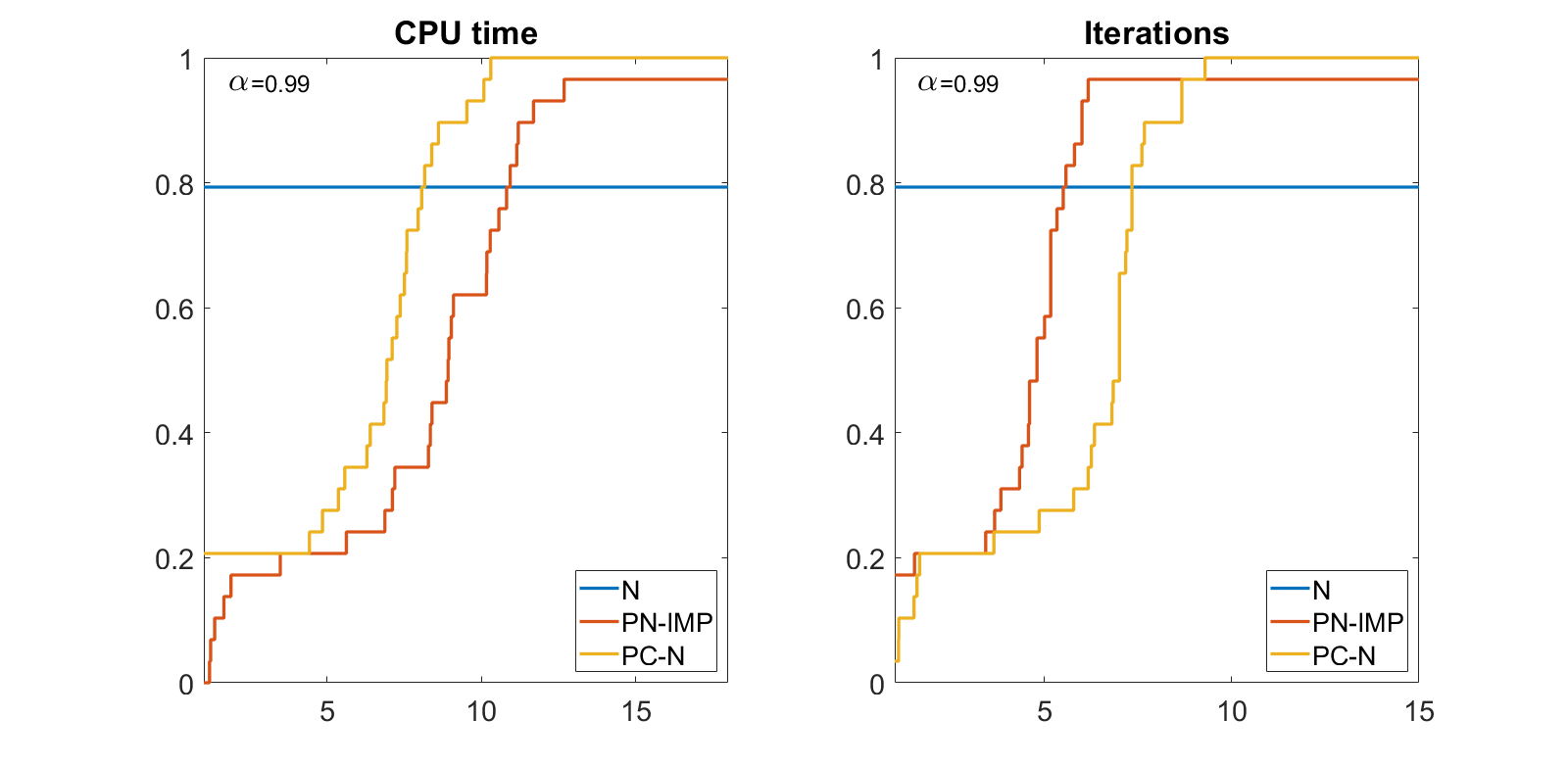}
\caption{Performance profiles for the 29 benchmark tensors and $\alpha= 0.99$.} \label{fig:perfprof3}
\end{figure}

The performance profile shows that the Newton method is always the fastest method, when it works, but it fails on various problems (6 out of 29 for $\alpha=0.99$). The new method is about a factor 5 slower, but has much higher reliability, failing on none of the problems. PN-IMP fails on one of the problems, and it is slower in terms of CPU time: this is to be expected, since eigenvalue computations are generally slower than QR factorizations and linear system solutions.

In the second experiment, we counted the number of failures of each of the three algorithms for different values of $\alpha$. The results are reported in Table~\ref{tab:failures}.

\begin{table}
\centering
\begin{tabular}{cccc}
    \toprule
    $\alpha$  & N & PN-IMP & PC-N \\
    \midrule
    0.90 & 102 & 22 & 0 \\
    0.95 & 317 & 47 & 0  \\
    0.99 & 693 & 136 & 0 \\
    \bottomrule
\end{tabular}
\caption{Number of failures recorded on 100k random sparse tensors of size $5$.} \label{tab:failures}
\end{table}

\section{Conclusions} \label{sec:conclusions}

The experimental results confirm that the suggested method, based on the combination of predictor-corrector continuation methods and Newton's method for underdetermined systems, is an effective and reliable method to solve multilinear Pagerank problems, even in cases that are problematic for most other methods. The theoretical analysis performed confirms that the solution curve always reaches a valid solution, at least in the case where the solution curve has no singular points. It remains to consider how this method scales to larger problems.

\paragraph{Acknowledgments} The authors are grateful to Beatrice Meini for several useful discussions on this topic.

\bibliography{biblio.bib}
\bibliographystyle{abbrv}

\end{document}